\documentclass[12pt, oneside, letterpaper, reqno]{amsart} 

\usepackage{latexsym}
\usepackage{amssymb}
\usepackage{amsmath}
\usepackage{mathrsfs}
\usepackage{colonequals}

\usepackage[marginparsep=0pt, 
            left=25mm, 
            right=25mm, 
            top=30mm, 
            bottom=30mm]{geometry}
\usepackage{setspace}
\allowdisplaybreaks

\usepackage{amsthm}
\theoremstyle{plain}
\newtheorem{X}{X}[section]
\newtheorem{lemma}[X]{Lemma}
\newtheorem{theorem}[X]{Theorem}
\newtheorem{proposition}[X]{Proposition}

\theoremstyle{definition}

\usepackage{hyperref}
\hypersetup{
    pdftoolbar=true,                        
    pdfmenubar=false,                        
    pdffitwindow=false,                      
    pdfstartview={FitH},                    
    pdftitle={},    
    pdfauthor={Tristan Freiberg, P\"ar Kurlberg},    
    pdfcreator={Tristan Freiberg, P\"ar Kurlberg},
    pdfproducer={Tristan Freiberg, P\"ar Kurlberg},
    pdfnewwindow=true,                       
    colorlinks=true,
    linkcolor={black},                      
    citecolor={black},
    filecolor={black},
    urlcolor={black},         
}
\usepackage{cite}

\numberwithin{equation}{section}

\renewcommand{\le}{\ensuremath{\leqslant}}
\renewcommand{\ge}{\ensuremath{\geqslant}}

\renewcommand{\Re}[1]{\ensuremath{\mathrm{Re}(#1)}}

\newcommand{\br}[1]{\ensuremath{\left(#1\right)}} 
\newcommand{\Br}[1]{\ensuremath{\left\{#1\right\}}}   

\newcommand{\ab}[1]{\ensuremath{\vert#1\vert}} 
\newcommand{\abs}[1]{\ensuremath{\left\lvert#1\right\rvert}} 

\newcommand{\sums}[2]{\ensuremath{\sum_{\substack{#1\\#2}}}}
\newcommand{\sumss}[3]{\ensuremath{\sum_{\substack{#1\\#2\\ #3}}}}

\newcommand{\BigO}[1]{\ensuremath{O\br{#1}}} 
\newcommand{\Oh}{\ensuremath{O}} 

\newcommand{\defeq}{\ensuremath{\colonequals}} 
\newcommand{\eqdef}{\ensuremath{\equalscolon}} 

\newcommand{\li}[1]{\ensuremath{\mathrm{Li}\br{#1}}} 

\newcommand{\dd}[1]{\ensuremath{\,\mathrm{d}#1}}  

\newcommand{\Ep}{\ensuremath{\tilde{E}}} 
\newcommand{\expep}{\ensuremath{e_{p}}}  

\newcommand{\Gal}[1]{\ensuremath{\mathrm{Gal}\br{#1/\Q}}}  
\renewcommand{\L}[1]{\ensuremath{L_{#1}}} 
\newcommand{\disc}[1]{\ensuremath{\Delta_{#1}}}  
\renewcommand{\deg}[1]{\ensuremath{n_{#1}}} 

\newcommand{\rad}{\ensuremath{\mathrm{rad}}}

\newcommand{\Z}{\ensuremath{\mathbb{Z}}} 
\newcommand{\Q}{\ensuremath{\mathbb{Q}}} 
\newcommand{\QQ}{\ensuremath{\overline{\mathbb{Q}}}} 
\newcommand{\R}{\ensuremath{\mathbb{R}}} 
\newcommand{\fp}{\ensuremath{\mathbb{F}_p}} 
\newcommand{\OO}{\ensuremath{\mathfrak{O}}} 

\title{On the average exponent of elliptic curves modulo $p$}

\author{Tristan Freiberg}
\address{Department of Mathematics, 
         KTH Royal Institute of Technology, 
         SE-100 44 Stockholm, Sweden}
\email{tristanf@kth.se}

\author{P\"ar Kurlberg}
\address{Department of Mathematics, 
         KTH Royal Institute of Technology, 
         SE-100 44 Stockholm, Sweden}
\email{kurlberg@kth.se}

\thanks{The first author is supported by the G\"oran Gustafsson 
        Foundation (KVA). The second author is supported by grants 
        from the G\"oran Gustafsson Foundation (KVA) and the 
        Swedish Research Council.}

\date{\today}

\begin{document}

\begin{abstract}
  Given an elliptic curve $E$ defined over $\Q$ and a prime $p$ of
  good reduction, let $\Ep(\fp)$ denote the group of $\fp$-points 
  of the reduction of $E$ modulo $p$, and let $\expep$ denote the 
  exponent of this group.  Assuming a certain form of the 
  Generalized Riemann Hypothesis (GRH), we study the average of 
  $\expep$ as $p \le X$ ranges over primes of good reduction, and 
  find that the average exponent essentially equals 
  $p\cdot c_{E}$, where the constant $c_{E} > 0$ depends on $E$.  
  For $E$ without complex multiplication (CM), $c_{E}$ can be 
  written as a rational number (depending on $E$) times a 
  universal constant, 
  $c \defeq \prod_{q}\br{1 - \frac{q^3}{(q^2-1)(q^5-1)}}$, 
  the product being over all primes $q$. 
  Without assuming GRH, we can determine the 
  average exponent when $E$ has CM, as well as give an upper bound 
  on the average in the non-CM case.
\end{abstract}

\maketitle

\section{Introduction}\label{Section 1}

Given an elliptic curve $E$ defined over $\Q$, and a prime $p$ for
which $E$ has good reduction, let $\Ep(\fp)$ denote the group of
$\fp$-points of the reduction of $E$ modulo $p$. The behavior of 
$\Ep(\fp)$ as $p$ varies over the primes has received considerable 
attention --- the oscillations of the cardinalities $|\Ep(\fp)|$ 
is a central question in modern number theory, and the structure 
of $\Ep(\fp)$ as a group, for example, the existence of large 
cyclic subgroups, especially of prime order, is of interest 
because of applications to elliptic curve cryptography 
\cite{K1987, M1986}.

If $p$ is a prime of good reduction then 
$\Ep(\fp) \cong \Z/d_p\Z \times \Z/e_p\Z$ for uniquely determined
integers $d_p, e_p$, with $d_p \mid e_p$. The size of the maximal 
cyclic subgroup, that is the exponent, of $\Ep(\fp)$ is therefore 
$e_p$. For primes $p$ of bad reduction we set $e_p = 0$. The 
purpose of this paper, motivated by a question of Joseph Silverman 
(personal communication), is to investigate the average of 
$\expep$ as $p$ varies. Conditional on a certain form of the 
Generalized Riemann Hypothesis (GRH), we will show that there 
exists $c_{E} \in (0,1)$ such that
\begin{align*}
\sum_{p\le X} \expep \sim c_{E} \cdot \li{X^2}
\quad \text{as} \quad X\to\infty,
\end{align*}
where $\li{X^2} \defeq \int_2^{X^2} \dd{t}/(\log t)$ is the 
logarithmic integral of $X^2$. Since 
$\sum_{p \le X} p \sim \li{X^2}$ (by partial summation and the 
prime number theorem), we may interpret this as the 
average value of $\expep$ being $p\cdot c_{E}$.

Before stating our main theorem we explain what we mean by GRH. 
Given a positive integer $k$, let $\L{k}$ denote the $k$-division 
field of $E$, that is, the number field obtained by adjoining to 
$\Q$ the coordinates of all points in $E[k]$, the subgroup of 
$k$-torsion of points of $E$. Let $\zeta_{\L{k}}(s)$ denote the 
Dedekind zeta function associated with $\L{k}$. We say that 
$\zeta_{L_k}(s)$ satisfies the Riemann Hypothesis (RH) if all 
zeros with positive real part lie on the line $\Re{s}=1/2$. By GRH 
we will here, and in what follows, mean that the Riemann 
Hypothesis holds for $\zeta_{\L{k}}$ for all positive integers 
$k$.

\begin{theorem}\label{Theorem 1.1}
  Given an elliptic curve $E$ defined over $\Q$, there exists a 
  number $c_{E} \in (0,1)$ such that on GRH we have
\begin{align*}
\sum_{p \le X} \expep 
 = c_E \cdot \li{X^2} 
   + \Oh_{E}\br{X^{19/10}(\log X)^{6/5}}
\end{align*}
for $X \ge 2$. The implied constant depends on $E$ at most. 
\end{theorem}

Settling for a weaker error term, we can remove the GRH assumption 
for CM-curves.

\begin{theorem}\label{Theorem 1.2}
Let $E$ be an elliptic curve defined over $\Q$ with complex 
multiplication, and let $c_E$ be as in Theorem \ref{Theorem 1.1}. 
For $X \ge 3$, we have
\begin{align*}
 \sum_{p \le X} \expep 
   = c_E \cdot \li{X^2} \cdot  
      \Br{1 + \Oh_E\br{\frac{\log\log X}{(\log X)^{1/8}}}}.
\end{align*}
The implied constant depends on $E$ at most. 
\end{theorem}

({\em Note added in proof:} The error terms in Theorems 
\ref{Theorem 1.1} and \ref{Theorem 1.2} have recently been 
improved by Wu \cite{W2012} and Kim \cite{K2012}. See Section 
\ref{Section 8} for details.)

For non-CM curves we can give an unconditional upper bound of the
correct order of magnitude. In the following theorem, we use the
notation $F(X) \lesssim G(X)$, which means that 
$\limsup_{X \to \infty} F(X)/G(X) \le 1$.

\begin{theorem}\label{Theorem 1.3}
Let $E$ be an elliptic curve defined over $\Q$, and let $c_E$ as 
in Theorem \ref{Theorem 1.1}. As $X$ tends to infinity, we have
\begin{align*}
\sum_{p \le X} \expep 
  \lesssim c_E \cdot \li{X^2}.
\end{align*}
\end{theorem}

We will now describe $c_{E}$ in more detail. With 
$\deg{\L{k}} \defeq [\L{k} : \Q]$ denoting the degree of the 
extension $\L{k}/\Q$, $\omega(k)$ the number of distinct prime 
factors of $k$, $\phi(k)$ the Euler totient function of $k$, and 
$\rad(k)$ the largest squarefree divisor of $k$, we have (whether 
or not $E$ has CM)
\begin{align}\label{(1.1)}
c_E \defeq  \sum_{k=1}^{\infty}
\frac{(-1)^{\omega(k)}\phi(\rad(k))}{k\deg{\L{k}}}.
\end{align}
In Lemma \ref{Lemma 3.5} below, we will show that this sum is 
absolutely convergent, and that $c_{E} \in (0,1)$. If $E$ does not 
have CM, there exists a universal constant 
\begin{align}\label{(1.2)}
c \defeq \prod_{q}\br{1 - \frac{q^3}{(q^2-1)(q^5-1)}} 
    = 0.8992282528\ldots,
\end{align}
such that $c_E/c$ is a rational number depending only on $E$. If 
$E$ has CM by an order $\OO$ in a imaginary quadratic number field 
$K$, $c_{E}$ can similarly be written as a rational number 
(depending on $E$) times an Euler product, depending only on $K$, 
of the form
\begin{align*}
\prod_{\text{$q$ splits in $K$}}
\br{1- \frac{1}{q^{2}(1-1/q)(1-1/q^{3})}}
\cdot
\prod_{\text{$q$ inert in $K$}}
\br{1- \frac{1}{q^{2}(1+1/q)(1-1/q^{3})}}.
\end{align*}
We will indicate how to prove the last two statements in 
Section \ref{Section 7}.

\subsection{Background and discussion}\label{Subsection 1.1}

{\em The multiplicative order of a number modulo $p$.}  Given a
rational number $g \ne 0, \pm 1$ and a prime $p$, let $l_{g}(p)$ 
denote the multiplicative order of $g$ modulo $p$ (unless 
$p$ divides $ab$, where $g = a/b$, $a,b$ coprime, in which case 
set $l_g(p) = 0$). In \cite{KP2010}, the second author and 
Pomerance, on assuming the Riemann hypothesis for Dedekind zeta 
functions associated with certain Kummer extensions, determined 
the average of $l_{g}(p)$ as $p \le X$ ranges over primes by 
showing that
\begin{align*}
\sum_{p\le X} l_{g}(p) 
= C_{g} \cdot \li{X^2} 
   + \BigO{\frac{X^2}{(\log X)^{3/2-2/\log\log\log X}}}, 
\end{align*}
where $C_{g}$ can be expressed in terms of the degrees of certain
Kummer extensions, namely
\begin{align*}
C_{g} \defeq
\sum_{k=1}^{\infty} 
     \frac{(-1)^{\omega(k)}\phi(k)\rad(k)}
          {k^{2}\cdot [\Q(g^{1/k},e^{2\pi i/k}):\Q]}
= \sum_{k=1}^{\infty} 
       \frac{(-1)^{\omega(k)}\phi(\rad(k))}
            {k \cdot [\Q(g^{1/k},e^{2\pi i/k}):\Q]}.
\end{align*}

Thus, even though we consider two rather different quantities
associated with groups modulo $p$, namely the multiplicative 
{\em order} of a fixed element modulo $p$ and the {\em exponent} 
of an elliptic curve modulo $p$, the sums defining $C_{g}$ and 
$c_{E}$ are very similar; the only difference is that degrees of 
Kummer fields replace degrees of $k$-divison fields. (Note that 
the exponent fluctuations for $(\Z/p\Z)^{\times}$ are essentially 
trivial since the group is cyclic.) Further, $C_{g}$ can also be 
written as the product of a rational number (depending on $g$) 
times a universal constant, namely 
$C \defeq \prod_q (1-q/(q^3-1)) = 0.5759599689\ldots$ (the product 
being over all primes $q$).

{\em Upper and lower bounds on $\expep$.} As $p \to \infty$, 
Hasse's bound implies that $|\Ep(\fp)|/p \sim 1$ which, together
with the rank of $\Ep(\fp)$ being at most two, implies that 
$\sqrt{p} \ll \expep \ll p$. For $E$ any non-CM curve, Schoof
\cite{S1991} improved the lower bound to 
$\expep \gg \sqrt{p} \cdot \log p/\log\log p$, and noted that this 
is unlikely to hold for CM curves since the curve $E$ defined by 
$y^{2} = x^{3}-x$ has exponent $\expep = \sqrt{p-1}$ for any prime 
of the form $p=(4n)^{2}+1$.

If one removes zero density subsets of the primes, Duke 
\cite{D2003} has significantly improved the lower bound. Namely, 
if $f: \R^{+} \to \R^+$ is any increasing function tending to 
infinity, $\expep > p/f(p)$ holds for `almost all' primes, in the 
sense that it holds for all but $o(\pi(X))$ primes $p \le X$. 
(As usual, $\pi(X)$ denotes the number of primes up to $X$.)
For CM curves the result is unconditional, whereas for non-CM 
curves GRH is assumed. (For the latter he also shows that the 
weaker bound $\expep > p^{3/4}/\log p$ holds unconditionally for 
almost all primes.)

Finally we mention that Shparlinski \cite{S2009} has shown that 
for any $\epsilon>0$ and $p$ large, $\expep \ge p^{1-\epsilon}$
holds for almost all elliptic curves $E$ in the family
$\{E_{a,b}\}_{a,b}$, where $E_{a,b}$ denotes the curve $y^{2} =
x^{3}+ax+b$.

{\em The proportion of primes for which $\Ep(\fp)$ is cyclic.} A 
question closely related to the size of the exponent is cyclicity
--- how often does the equality $|\Ep(\fp)| = \expep$ hold? 
Borosh, Moreno and Porta \cite{BMP1975} conjectured that 
$\Ep(\fp)$ is cyclic for infinitely many primes $p$, except in 
certain cases where this cannot be so for `trivial reasons'. Serre 
later proved \cite{S1978}, on GRH, that 
\begin{align}\label{(1.3)}
\begin{split}
\frac{1}{\pi(X)}
\sideset{}{^*}
\sum_{\substack{p \le X \\ \text{$\Ep(\fp)$ is  cyclic}}} 1 
\sim c^{*}_E \quad \text{as} \quad X\to\infty,
\end{split}
\end{align}
where $\textstyle\sum^{*}$ denotes a sum restricted to $p$ at 
which $E$ has good reduction, and, with $\mu(k)$ denoting the 
M\"obius function of $k$,
$c^{*}_E = \sum_{k=1}^{\infty} \mu(k)/\deg{\L{k}}$. Furthermore, 
$c^{*}_E > 0$ {\em unless} all $2$-torsion points on $E$ are 
defined over $\Q$, an obvious obstruction\footnote{The only way 
for $E(\Q)$ to have a rationally defined subgroup isomorphic to
$\Z/\ell\Z \times \Z/\ell\Z$ is if $\ell=2$; this easily follows 
from the fact that $E(\R)$ is isomorphic to either $\R/\Z$ or 
$\Z/2\Z \times \R/\Z$.} to $\Ep(\fp)$ being cyclic.

Cojocaru and Murty \cite{CM2004} obtained versions of 
\eqref{(1.3)} with effective error terms, and in the special case 
in which $E$ has CM, Murty \cite{M1983} was quite remarkably able 
to establish \eqref{(1.3)} {\em unconditionally} (the proofs were 
later significantly simplified by Cojocaru~\cite{C2003}).

For more background on this and related topics, we recommend the 
nice survey article \cite{C2004} by Cojocaru.

\section{Outline of the proof of Theorem 1.1}\label{Section 2}

We begin by noting that our approach is in spirit a synthesis of 
the ideas in \cite{S1978, KP2010}, together with refinements by 
Murty \cite{M1983} and Cojacaru \cite{C2003}.

As for notation, in this outline we shall use `$\approx$' to 
indicate equality with an implied error term, and $p$ shall always 
denote a prime of good reduction. Recall that 
$\expep = |\Ep(\fp)|/d_p$, so if $|\Ep(\fp)| \eqdef p + 1 - a_p$,
then, since $\ab{a_p} \le 2\sqrt{p}$ by Hasse's inequality, we 
have
\begin{align}\label{(2.1)}
\sum_{p \le X} \expep \approx \sum_{p \le X} \frac{p}{d_p}.
\end{align}
We can treat the sum $\sum_{p \le X} p/d_p$ by using partial 
summation and the prime number theorem, once we have evaluated the 
sum $\sum_{p \le X} \frac{1}{d_p}$. 

Since $d_p \mid e_p$ we have 
$d_p^2 \le |\Ep(\fp)| \le (\sqrt{p} + 1)^2$ by Hasse's inequality, 
hence $d_p < 2\sqrt{X}$ for $p \le X$. As in \cite{KP2010}, we use 
the elementary identity 
$\frac{1}{k} = \sum_{hj \mid k} \frac{\mu(h)}{j}$ to write
\begin{align}\label{(2.2)}
\sum_{p \le X} \frac{1}{d_p} 
  = \sum_{p \le X} \sum_{hj \mid d_p} \frac{\mu(h)}{j} 
    = \sum_{hj \le 2\sqrt{X}} \frac{\mu(h)}{j} 
         \sums{p \le X}{hj \mid d_p} 1.
\end{align}
Now, by Lemma \ref{Lemma 3.1}, a positive integer $k$ divides 
$d_p$ if and only if $p$ splits completely in $\L{k}$, and the 
Chebotarev density theorem (cf.~Lemma \ref{Lemma 3.3}) then 
gives 
\begin{align*}
\sums{p \le X}{k \mid d_p} 1 
  \approx
    \sumss
       {p \le X}
        {\text{$p$ splits completely}}
         {\text{in $\L{k}/\Q$}} 1 
   \approx \frac{\li{X}}{\deg{\L{k}}}.
\end{align*}
Thus, 
\begin{align}\label{(2.3)}
\sum_{hj \le 2\sqrt{X}} \frac{\mu(h)}{j} 
  \sums{p \le X}{hj \mid d_p} 1 
 \approx \li{X}
           \sum_{hj \le 2\sqrt{X}} 
             \frac{\mu(h)}{j}\cdot \frac{1}{\deg{\L{hj}}} 
 \approx \li{X} 
           \sum_{k=1}^{\infty} 
             \frac{(-1)^{\omega(k)}\phi(\rad(k))}{k\deg{\L{k}}}.
\end{align}

The last error term in \eqref{(2.3)}, and indeed showing that the 
last sum is absolutely convergent, involves bounding a sum of the 
type $\sum_{k > Y} \frac{1}{\deg{\L{k}}}$. We do this in 
Lemma \ref{Lemma 3.4}, but here lower bounds for $\deg{\L{k}}$ are 
crucial.  

To give lower bounds on $\deg{\L{k}}$, we use Serre's open image
theorem \cite{S1968, S1972}.  If $E$ is a non-CM curve, 
compactness together with the image being open gives that the 
image of the absolute Galois group has finite index inside
$\mathrm{GL}_{2}(\hat{\Z})$, hence 
$\deg{\L{k}} \gg_{E} \phi(k)k^{3}$. If $E$ is a CM curve, a 
similar open image result of Serre gives that 
$\deg{\L{k}} \gg \phi(k)^2$. (For more details, see 
Proposition \ref{Proposition 3.2}.)

Now, combining \eqref{(2.1)}, \eqref{(2.2)}, and \eqref{(2.3)}, we
obtain, via partial summation and the prime number theorem, that
\begin{align*}
\sum_{p \le X} \expep 
  \approx \li{X^2} \cdot
    \sum_{k=1}^{\infty} 
       \frac{(-1)^{\omega(k)}\phi(\rad(k))}{k\deg{\L{k}}}, 
\end{align*}
which is the claimed main term.

As for estimating the error terms, a slight complication arises 
--- even assuming GRH, we cannot directly bound the sum of the 
error terms in the Chebotarev density theorem for `large' $k$, 
that is, $k \in (\sqrt{X}/(\log X)^2, 2\sqrt{X}]$.  To deal with 
this range Serre used the fact that the cyclotomic field 
$\Q(e^{2\pi i/q})$ is contained in $\L{q}$; the sum can thus be 
restricted to primes $p \equiv 1 \bmod q$, and Brun's sieve is 
then enough to bound the errors arising from the large $k$. 
However, to make an exponent saving in the error term we use a 
refinement of Serre's approach due to Cojacaru and Murty 
\cite{CM2004} (see Lemma~\ref{Lemma 3.6} for further details.)

\section{Preliminaries}\label{Section 3}

In this section we collect some needed results on elliptic curves.

\subsection*{Notation} Throughout, $p$, $q$, and $\ell$ denote 
(rational) primes; $h$, $j$, $k$, and $m$ denote positive 
integers. The arithmetic functions $\omega$, $\phi$, $\rad$, and 
$\mu$ have already been introduced; also, $\tau(k)$ is the number 
of divisors of $k$, and $\sigma(k)$ is the sum of the divisors of 
$k$.  

Whenever we write $F = \BigO{G}$, $F \ll G$, or $G \gg F$, we mean 
that $|F| \le c\cdot G$ where $c$ is an absolute positive 
constant. By $F \asymp G$ we mean that $F \ll G \ll F$.

The logarithmic integral is defined for numbers $t \ge 2$
by $\li{t} \defeq \int_2^t \frac{\dd{u}}{\log u}$.

We fix an elliptic curve $E$, defined over $\Q$, of conductor $N$. 
The results in this section relate to $E$. For primes $p$ of good
reduction (that is, $p \nmid N$), $d_p$ and $e_p$ are the unique 
positive integers such that we have 
$\Ep(\fp) \cong \Z/d_p\Z \times \Z/e_p\Z$ with $d_p \mid e_p$. 
Thus $e_p$ is the exponent of $\Ep(\fp)$ if $p \nmid N$, and we 
set $e_p = 0$ if $p \mid N$. Also, if $p \nmid N$, 
$a_p \defeq p + 1 - |\Ep(\fp)|$, and $\pi_p$ denotes a root of the 
polynomial $X^2 - a_pX + p \in \Z[X]$.

The $k$-division field of $E$ is denoted $\L{k}$; $\deg{\L{k}}$ 
denotes the degree of the extension $\L{k}/\Q$, and $\disc{\L{k}}$ 
denotes its discriminant. 

\begin{lemma}\label{Lemma 3.1}
If $p \nmid kN$ then the following statements are equivalent.
\begin{enumerate}
 \item $\Ep(\fp)$ contains a subgroup isomorphic to
       $\mathbb{Z}/k\mathbb{Z} \times \mathbb{Z}/k\mathbb{Z}$.
 \item $p$ splits completely in $\L{k}$.
 \item $\frac{\pi_p - 1}{k}$ is an algebraic integer.
 \end{enumerate}
\end{lemma}
\begin{proof}
For the equivalence of (1) and (2), see \cite[Lemma 2]{M1983}. 
For the equivalence of (2) and (3), see \cite[Lemma 2.2]{C2003}.
\end{proof}

We now give some estimates on the degree of the $k$-division field 
of $E$.

\begin{proposition}\label{Proposition 3.2}
\noindent (a) $\L{k}$ contains $\Q(e^{2\pi i/k})$ (the $k$-th 
cyclotomic field), hence $p$ splits completely in $\L{k}$ only if 
$p \equiv 1 \bmod k$. Also, $\phi(k)$ divides $\deg{\L{k}}$. \\
\noindent (b) $\deg{\L{k}}$ divides 
$|\mathrm{GL}_2(\Z/k\Z)| 
  = k^{4}\prod_{q \mid k} \br{1 - 1/q}\br{1 - 1/q^2}$. \\
\noindent (c) If $E$ is a non-CM curve, then there exists a 
constant $B_E \ge 1$, depending only on $E$, such that
$|\mathrm{GL}_2(\Z/k\Z)| \le B_E\cdot \deg{\L{k}}$ for every $k$. 
\\
\noindent (d) If $E$ has CM, then 
$\phi(k)^2 \ll \deg{\L{k}} \le k^2$.
\end{proposition}

\begin{proof}
(a) See \cite[Corollary 8.1.1, Chapter III]{S1986}.

(b) First of all note that $\deg{\L{k}} = |\Gal{\L{k}}|$ as 
$\L{k}/\Q$ is a Galois extension. The action of the absolute 
Galois group $\Gal{\QQ}$ on the $k$-torsion subgroup $E[k]$ 
of $E$ induces a representation
$\rho_{E,k} : \Gal{\QQ} 
  \to \mathrm{Aut}(E[k])
    \cong \mathrm{GL}_2(\mathbb{Z}/k\mathbb{Z})$,
which is injective.

(c) With $T_{\ell}(E)$ denoting the $\ell$-adic Tate module
of $E$, the action of $\Gal{\QQ}$ on 
$\prod_{\ell}\mathrm{Aut}(T_{\ell}(E)) 
  \cong \prod_{\ell} \mathrm{GL}_{2}(\Z_{\ell})$ 
induces a representation 
$\rho_{E} : \Gal{\QQ} 
  \to \prod_{\ell}\mathrm{GL}_{2}(\Z_{\ell})$.  
By Serre's open image theorem \cite[Theorem 3]{S1972}, the image
of $\rho_{E}$ is open, and since 
$\prod_{\ell} \mathrm{GL}_{2}(\Z_{\ell})$ is compact, the image is  
of finite index. Since $\Gal{\L{k}}$ is isomorphic to the 
projection of $\mathrm{Im}(\rho_{E})$, by the map 
$\prod_{\ell} \mathrm{GL}_{2}(\Z_{\ell}) 
   \to \mathrm{GL}_{2}(\Z/k\Z)$, the result follows.

(d) If $E$ has CM, Serre's open image result for the CM case 
\cite[Corollaire, Th\'eor\`eme 5, p.302]{S1972} gives that 
$\deg{\L{k}} \gg \phi(k)^2$.  (If $E$ has CM by an order 
$\OO$, the image is open in $\prod_{\ell} \OO_{\ell}^{\times}$, 
where $\OO_{\ell} = \OO \otimes \Z_{\ell}$.) The upper bound 
follows from $|(\OO/k\OO)^{\times}| \le k^{2}$.
\end{proof}

\begin{lemma}\label{Lemma 3.3}
(a) There exist absolute constants $c_1, B > 0$, such that the 
following statements hold. 
(i) If $c_1k^{14}N^2 \le \log X$, then, whether or not $E$ has CM, 
\begin{align}\label{(3.1)}
|\{ \text{$p \le X$ : $p \nmid N$ and $k \mid d_p$}\}|  
  = \frac{\li X}{\deg{\L{k}}} 
   + \BigO{X\exp\br{-B(\log X)^{5/14}}}.
\end{align}
(ii) If $c_1k^{8}N^2 \le \log X$ and $E$ has CM, then
\begin{align}\label{(3.1b)}
|\{ \text{$p \le X$ : $p \nmid N$ and $k \mid d_p$}\}|  
  = \frac{\li X}{\deg{\L{k}}} 
   + \BigO{X\exp\br{-B(\log X)^{3/8}}}.
\end{align}

(b) Whether or not $E$ has CM, for $X \ge 2$ we have, on GRH, that
\begin{align}\label{(3.2)}
 |\{ \text{$p \le X$ : $p \nmid N$ and $k \mid d_p$}\}| 
   = \frac{\li X}{\deg{\L{k}}} 
     + \BigO{X^{1/2}\log(XN) }.
\end{align}
\end{lemma}

\begin{proof} 
Note that if $p \le X$ and $p \nmid N$, then {\em a priori} we 
have $k \le 2\sqrt{X}$ for $k \mid d_p$, because $d_p \mid e_p$ 
and so $d_p^2 \le |\Ep(\fp)| \le (\sqrt{p} + 1)^2$ by Hasse's 
inequality. Since $p \nmid d_p$ 
(cf.~\cite[Exercise 5.6(a), p.145]{S1986}), the conditions 
$p \nmid N$ and $k \mid d_p$ are equivalent to $p \nmid kN$ and 
$k \mid d_p$, that is $p \nmid kN$ and $\Ep(\fp)$ contains a 
subgroup isomorphic to $\Z/k\Z \times \Z/k\Z$. Therefore, by 
Lemma \ref{Lemma 3.1}, 
\begin{align}\label{(3.3)}
 \sumss{p \le X}
       {\text{$p$ \rm{splits completely}}}
       {\text{\rm{in} $\L{k}/\Q$}} 1  
  = |\{ \text{$p \le X$ : $p \nmid N$ and $k \mid d_p$}\}| 
    + \BigO{\log(XN)},
\end{align}
where the $\BigO{\log(XN)}$ term is the negligible contribution 
from the primes $p$ dividing $kN$: 
$\omega(kN) \ll \log(kN) \ll \log(XN)$.

(a) (i) As $\L{k}/\Q$ is a Galois extension, an effective form of 
the Chebotarev density theorem (for example, see 
\cite[Lemma 2]{M1984}) gives
\begin{align}\label{(3.4)}
\sumss{p \le X}
      {\text{$p$ \rm{splits completely}}}
      {\text{\rm{in} $\L{k}/\Q$}} 1
 = \frac{\li X}{\deg{\L{k}}} 
 + \BigO{X\exp\br{-B\sqrt{\textstyle\frac{\log X}{\deg{\L{k}}}}}},
\end{align}
where $B$ is an absolute positive constant, provided
 \begin{align}\label{(3.5)}
 c\cdot \max\Br{\deg{\L{k}}\ab{\disc{\L{k}}}^{2/\deg{\L{k}}}, 
 \deg{\L{k}}\br{\log \ab{\disc{\L{k}}}}^2}
   \le \log X,
 \end{align}
for a certain absolute positive constant $c$. We claim that
there is an absolute positive constant $c_1$ such that if 
$c_1k^{14}N^2 \le \log X$, then \eqref{(3.5)} does indeed
hold. In that case, Proposition \ref{Proposition 3.2}(b) gives 
$\sqrt{\deg{\L{k}}} \le k^2 \le (\log X)^{1/7}$ 
(we may suppose that $c_1 \ge 1$), so
applying this to the error term in \eqref{(3.4)} 
and combining with \eqref{(3.3)} gives \eqref{(3.1)}. 

We now prove our claim. The first of the following sequence
of inequalities holds with any Galois extension $L/\Q$
in place of $\L{k}/\Q$ \cite[Proposition 6, p.130]{S1981};
the second holds because the ramified primes of 
$\L{k}/\Q$ are divisors of $kN$
\cite[Proposition 4.1(a), Chapter VII]{S1986}:
\begin{align}\label{(3.6)}
 \frac{\log \ab{\disc{\L{k}}}}{\deg{\L{k}}} 
  \le \log \deg{\L{k}} 
    + \hspace{-8pt} 
       \sums{\text{$p$ ramifies in}}
            {\L{k}/\Q} \hspace{-8pt} \log p 
    \le \log \deg{\L{k}} + \sum_{p \mid kN} \log p 
     \le \log \deg{\L{k}} + \log(kN).
\end{align}
Our claim follows straightforwardly using this and the inequality 
$\deg{\L{k}} \le k^4$.

(ii) Similar, but in the CM case we use the fact that 
$\deg{\L{k}} \le k^2$, by Proposition \ref{Proposition 3.2}(d).

(b) By \cite[Th\'eor\`eme 4, p.133]{S1981}, on GRH we have
\begin{align*}
 \sumss{p \le X}
       {\text{$p$ \rm{splits completely}}}
       {\text{\rm{in} $\L{k}/\Q$}} 1  
  = \frac{\li X}{\deg{\L{k}}} 
   + \BigO{X^{1/2}
        \br{\frac{\log \ab{\disc{\L{k}}}}{\deg{\L{k}}} + \log X}}.
\end{align*}
Applying \eqref{(3.6)} and the inequality 
$\deg{\L{k}} \le k^4 \ll X^2$, we obtain \eqref{(3.2)} by putting 
this into \eqref{(3.3)}.
\end{proof}

\begin{lemma}\label{Lemma 3.4}
 With $B_E$ as in Proposition \ref{Proposition 3.2}(c), we have
\begin{align}\label{(3.7)}
 \sum_{k > Y} \frac{1}{\deg{\L{k}}} 
   \ll 
     \begin{cases}
       1/Y   
         & \text{if $E$ has CM}, \\
       B_E/Y^3 
         & \text{if $E$ is a non-CM curve},
     \end{cases}
\end{align}
and
\begin{align}\label{(3.8)}
 \sum_{k > Y} \frac{\sigma(k)\tau(k)}{k\deg{\L{k}}} 
   \ll 
     \begin{cases}
       (\log Y)/Y
         & \text{if $E$ has CM}, \\
       B_E(\log Y)/Y^3
         & \text{if $E$ is a non-CM curve}.
     \end{cases}         
\end{align}
\end{lemma}
\begin{proof}
In the CM case we have, by Proposition \ref{Proposition 3.2}(d), 
that
\begin{align}\label{(3.9)}
 \sum_{k > Y} \frac{1}{\deg{\L{k}}} 
  \ll \sum_{k > Y} \frac{1}{\phi(k)^2} 
   \ll \frac{1}{Y}.
\end{align}
The last bound holds because $\phi(k) \approx k$ `on average'. 
It can be proved by entirely elementary means, the key being
the identity 
$\frac{k}{\phi(k)} = \sum_{j \mid k} \frac{|\mu(j)|}{\phi(j)}$.
We spare the reader the details. Similarly, since 
$|\mathrm{GL}_2(\Z/k\Z)| 
 = k^{4}\prod_{q \mid k} \br{1 - 1/q}\br{1 - 1/q^2}
  \gg k^3\phi(k)$,
Proposition \ref{Proposition 3.2}(c) gives
\begin{align*}
 \sum_{k > Y} \frac{1}{\deg{\L{k}}}  
    \ll B_E \sum_{k > Y} \frac{1}{k^3\phi(k)}
      \ll \frac{B_E}{Y^3}
\end{align*}
in the non-CM case.

Again if $E$ has CM, we have
\begin{align*}
 \sum_{k > Y} \frac{\sigma(k)\tau(k)}{k\deg{\L{k}}} 
  \ll \sum_{k > Y} \frac{\sigma(k)\tau(k)}{k\phi(k)^2} 
   \ll \frac{\log Y}{Y}.
\end{align*}
One way to obtain the last bound is to establish that
$\sum_{k \le Y} \sigma(k)\tau(k) \ll Y^2\log Y$, then apply 
partial summation to show that 
$\sum_{k > Y} \sigma(k)\tau(k)/k^3 \ll (\log Y)/Y$, and use the 
identity 
$\frac{k}{\phi(k)} = \sum_{j \mid k} \frac{|\mu(j)|}{\phi(j)}$ two
more times. Similarly, if $E$ is a non-CM curve, then
\begin{align*}
 \sum_{k > Y} \frac{\sigma(k)\tau(k)}{k\deg{\L{k}}} 
  \ll B_E \sum_{k > Y} \frac{\sigma(k)\tau(k)}{k^4\phi(k)} 
   \ll \frac{B_E\log Y}{Y^3}.
\end{align*}
\end{proof}

\begin{lemma}\label{Lemma 3.5}
The sum in \eqref{(1.1)} defining $c_E$ is absolutely convergent, 
and $c_E \in (0,1)$. 
\end{lemma}
\begin{proof}
Absolute convergence of the sum in \eqref{(1.1)} follows at once 
from \eqref{(3.7)}. To show that $c_E \in (0,1)$, first note that 
for every $j \ge 1$, 
$\sum_{h \ge 1}  \frac{\mu(h)}{\deg{\L{hj}}} \ll_E 1$ 
by \eqref{(3.7)}, because $\deg{\L{hj}} \le \deg{\L{h}}$. Next, we 
claim that
\begin{align}\label{(3.10)}
  \sum_{h \ge 1}  \frac{\mu(h)}{\deg{\L{hj}}} \ge 0 
   \qquad (j \ge 1).
\end{align}
To see this, fix any $j \ge 1$ and any 
$Y \ge j$. Let $Q = \prod_{q \le Y} q$ and let $X$ be large enough 
in terms of $Y$ and $N$ so that $\log X \ge c_1(QY)^{14}N^2$, 
where $c_1$ is the constant of Lemma \ref{Lemma 3.3}(a). Also 
assume that $\sum_{h \mid Q} |\mu(h)| = 2^{Y} \ll \log X.$ An 
inclusion-exclusion argument gives
\begin{align*}
\sums{p \le X}{\text{$p \nmid N$, $d_p = j$}} 1 \le 
\sum_{h \mid Q} \mu(h) 
\sums{p \le X}{\text{$p \nmid N$, $hj \mid d_p$}}  1.
\end{align*}
Applying Lemma \ref{Lemma 3.3}(a)(i), we obtain
\begin{align*}
  \sums{p \le X}{\text{$p \nmid N$, $d_p = j$}} 1 \le 
\li{X} 
\sum_{h \mid Q} \frac{\mu(h)}{\deg{\L{hj}}}
+ \BigO{2^YX\exp\br{-B(\log X)^{5/14}}}.
\end{align*}
Dividing by $\li{X}$ and letting $X$ tend to infinity, we find 
that
\begin{align*}
0 \le  \limsup_{X \to \infty} \frac{1}{\li{X}}
  \sums{p \le X}{\text{$p \nmid N$, $d_p = j$}} 1
\le 
\sum_{h \mid Q} \frac{\mu(h)}{\deg{\L{hj}}}.
\end{align*}
Letting $Y$ tend to infinity (cf.~\eqref{(3.7)}), we obtain 
\eqref{(3.10)}.

Now, for any $Y \ge 1$,
\begin{align*}
  \sum_{j \le Y} \sum_{h \ge 1} 
  \frac{\mu(h)}{\deg{\L{hj}}}
  & = \sum_{k \ge 1} \frac{1}{\deg{\L{k}}}
     \sums{hj = k}{j \le Y} \mu(h) 
   = \sum_{k \le Y} \frac{1}{\deg{\L{k}}} 
     \sum_{hj = k} \mu(h)
  + \BigO{\sum_{k > Y} \frac{1}{\deg{\L{k}}} 
     \sum_{hj = k} 1 }. 
\end{align*}
Since $\sum_{hj = k} \mu(h)$ vanishes unless $k = 1$, the main 
term here is just $1/\deg{\L{1}} = 1$. Applying \eqref{(3.8)} to 
the $\Oh$-term (note that 
$\sum_{hj=k} 1 = \tau(k) \le \sigma(k)\tau(k)/k$), then letting 
$Y$ tend to infinity, we obtain
\begin{align}\label{(3.11)}
  \sum_{j \ge 1} \sum_{h \ge 1} \frac{\mu(h)}{\deg{\L{hj}}} = 1.
\end{align}

Now since 
$\sum_{h \mid k} \mu(h) \cdot h = (-1)^{\omega(k)}\phi(\rad(k))$, 
we have
\begin{align}\label{(3.12)}
 c_E
=  \sum_{k \ge 1} \frac{1}{\deg{\L{k}}}
   \sum_{hj = k} \frac{\mu(h)}{j}
=  \sum_{j \ge 1} \frac{1}{j} 
       \sum_{h \ge 1} \frac{\mu(h)}{\deg{\L{hj}}},
\end{align}
convergence being assured by \eqref{(3.7)}, \eqref{(3.10)} and 
\eqref{(3.11)}. In view of \eqref{(3.10)}, \eqref{(3.11)} and 
\eqref{(3.12)}, we see that $0 < c_E \le 1$. In fact recalling 
that $c_E^{*} = \sum_{h \ge 1} \frac{\mu(h)}{\deg{\L{h}}}$ is 
the cyclicity constant, we can deduce from 
\eqref{(3.10)} --- \eqref{(3.12)} that 
$c_E \in (0,1] \cap [c_E^{*},\frac{1}{2}(c_E^{*} + 1)]$, with 
$c_E = 1$ if and only if $c_E^{*} = 1$. 
However, that $c_E^{*} < 1$ can be seen by considering 
$\{\text{$p \le X$ : $p \nmid N$ and $q > t$}\}$. By the 
Chebotarev density theorem (cf.~Lemma \ref{Lemma 3.3}(a)(i)) we 
have, for $t \ge 2$ and sufficiently large $X$,  
\[
\frac{1}{\pi(X)}\cdot 
 |\{p \le X : p \nmid N, q \mid d_p \Rightarrow q > t\}|
 \le 1 - \frac{1}{\pi(X)}\cdot 
     |\{\text{$p \le X$ : $p \nmid N$ and $2 \mid d_p$}\}|   
  < 1 - \frac{1}{2\deg{\L{2}}}. 
\]
On the other hand, using inclusion-exclusion, followed by 
Lemma \ref{Lemma 3.3}(a)(i) and \eqref{(3.7)}, one can show that 
\[
\frac{1}{\pi(X)}\cdot 
|\{p \le X : p \nmid N, q \mid d_p \Rightarrow q > t\}|
= c_E^{*}
 + \Oh_{E}\br{\frac{1}{t}}
  + \Oh_{E}\br{2^t\exp\br{-B(\log X)^{5/14}}}.
\]
For suitable $t = t(X)$ and sufficiently large $X$, comparing 
gives $c_E^{*} < 1$. 
\end{proof}

\begin{lemma}\label{Lemma 3.6}
(a) For $X,Y \ge 2$ we have 
\begin{align}\label{(3.13)}
  |\{\text{$p \le X$ : $p \nmid N$ and $d_p > Y$} \}| 
     \ll \frac{X^{3/2}}{Y^2} + X^{1/2}\log X.
\end{align}
(b) If $E$ has CM and $2 < Y \le \log X$, then
\begin{align}\label{(3.14)}
 |\{\text{$p \le X$ : $p \nmid N$ and $d_p > Y$} \}| 
   \ll \frac{X\log\log X}{Y\log X}.
\end{align}
\end{lemma}
\begin{proof}
(a) First of all note that since $d_p \mid e_p$, we have 
$d_p^2 \le |\Ep(\fp)| \le (\sqrt{p} + 1)^2$ by Hasse's inequality.
Thus, if $p \nmid N$ and $d_p > Y$ then $d_p = k$ for some 
$Y < k \le 2\sqrt{X}$. But $d_p = k$ implies 
$k^2 \mid (p + 1 - a_p)$, and also $k \mid p - 1$ by 
Lemma \ref{Lemma 3.1} and Proposition \ref{Proposition 3.2}(a); 
hence $k \mid a_p - 2$. Since $a_p \le 2\sqrt{p}$ by Hasse's 
inequality, we therefore have
\begin{align*}
 \sums{p \le X}{\text{$p \nmid N$, $d_p > Y$}} 1
 & \le \sum_{Y < k \le 2\sqrt{X}} 
   \sumss{|a| \le 2\sqrt{X}}{a \ne 2}{a \equiv 2 \bmod k}  
     \hspace{5pt} 
       \sumss{p \le X}{a_p = a}{k^2 \mid p + 1 - a} 1
   + \sum_{Y < k \le 2\sqrt{X}} 
       \sumss{p \le X}{a_p = 2}{k^2 \mid p - 1} 1 \\
 & \ll \sum_{Y < k \le 2\sqrt{X}} 
         \frac{\sqrt{X}}{k}\br{\frac{X}{k^2} + 1} 
       + \sum_{Y < k \le 2\sqrt{X}} \br{\frac{X}{k^2} + 1} \\
 & \ll \frac{X^{3/2}}{Y^2} 
         + \sqrt{X}\log X 
           + \frac{X}{Y} 
             + \sqrt{X}
\ll  \frac{X^{3/2}}{Y^2} 
         + \sqrt{X}\log X .
\end{align*}
(Here we have used the elementary bound 
$\sum_{k > Y} k^{-m} \ll Y^{1-m}$, $m \ge 2$.)

(b) Suppose $E$ has complex multiplication by an order in the 
imaginary quadratic field $K = \Q(\sqrt{-D})$, $D$ a squarefree 
positive integer. We begin with the following observation. Since 
$p \nmid d_p$, the statement `$p \nmid N$ and $k \mid d_p$' is 
equivalent to the statement `$p \nmid kN$ and $p$ splits 
completely in $\L{k}$' (Lemma \ref{Lemma 3.1}). In that case 
$p \equiv 1 \bmod k$ by Proposition \ref{Proposition 3.2}(a). If 
$a_p = 0$, then we also have $p \equiv -1 \bmod k$, because 
$k^2 \mid d_pe_p = (p + 1 - a_p)$. But we can only have both
$p \equiv  -1 \bmod k$ and $p \equiv 1 \bmod k$ if $k = 1$ or
$2$. Therefore, we  
necessarily have $a_p \ne 0$ when $k \mid d_p$ and $k \ge 3$. 
Moreover, $\frac{\pi_p - 1}{k}$ is an algebraic integer 
(Lemma \ref{Lemma 3.1}), and since $\Q(\pi_p) = K$ when 
$a_p \ne 0$ (see \cite[Lemma 2.3]{C2003}), we have that
$\frac{\pi_p - 1}{k} \in \OO_K$.

Now, 
\begin{align}\label{(3.15)}
\{\text{$p \le X$ : $p \nmid N$ and $d_p > Y$} \} 
  \subseteq \mathscr{P}_1(X) 
            \cup \mathscr{P}_2(X,Y) 
            \cup \mathscr{P}_3(Y),
\end{align}
where
\begin{align*}
\mathscr{P}_1(X)   
 & \defeq 
  \{ \text{$p \le X$ : $p \nmid N$ and $q \mid d_p$ 
           for some $q \in (\log X, 2\sqrt{X}]$} \}, \\ 
\mathscr{P}_2(X,Y) 
 & \defeq 
  \{ \text{$p \le X$ : $p \nmid N$ and $q \mid d_p$ 
           for some $q \in (Y, \log X]$} \}, \\
\mathscr{P}_3(Y)   
 & \defeq 
  \{ \text{$p \le X$ : $p \nmid N$, $d_p > Y$, and 
           $q \mid d_p \Rightarrow q \le Y$} \} \\
 &\phantom{:} \subseteq  
  \{\text{$p \le X$ : $p \nmid N$, $k \mid d_p$ 
          for some $k \in [Y,Y^2]$}\}.
\end{align*}
Since $2 < Y \le \log X$ we have, by our initial observation, that 
\begin{align}\label{(3.16)}
 |\mathscr{P}_1(X)|  
= 
 \sum_{\log X < q \le 2\sqrt{X}}
 \sumss{p \le X}{\text{$p \nmid N$, $q \mid d_p$}}{a_p \ne 0} 1 
\le 
 \sum_{\log X < q \le 2\sqrt{X}}
 \sums{p \le X}{\frac{\pi_p - 1}{k} \in \OO_K} 1.
\end{align}
Since $\pi_p$ has norm $p$ in $K/\Q$, it follows that
\begin{align}\label{(3.17)}
|\mathscr{P}_1(X)|  
\le \sum_{\log X < q \le 2\sqrt{X}} |S_j(X;D,q)|,
\end{align}
with
\begin{align*}
S_j(X;D,k) 
  & \defeq 
    \Br{\text{$p \le X : p = \br{\textstyle\frac{u}{j}k+1}^2 
      + D\br{\textstyle\frac{v}{j}}^2k^2$ for some $u,v \in \Z$}},
\end{align*}
and $j = 1$ if $-D \equiv 2,3 \bmod 4$; $j = 2$ if 
$-D \equiv 1 \bmod 4$. A trivial bound for $|S_i(X;D,q)|$ will 
suffice here:
\begin{align}\label{(3.18)}
\begin{split}
\sum_{\log X < q \le 2\sqrt{X}} |S_j(X;D,q)| 
 & \ll  \sum_{\log X < q \le 2\sqrt{X}} 
        \frac{\sqrt{X}}{q\sqrt{D}}\br{\frac{\sqrt{X}}{q} + 1} \\
 & \ll X\sum_{q > \log X}  \frac{1}{q^2} 
    + \sqrt{X}\sum_{q \le 2\sqrt{X}} \frac{1}{q} \\
 & \ll \frac{X}{(\log X)(\log\log X)} + \sqrt{X}\log\log X.
\end{split}
\end{align}
(To obtain the last bound, apply partial summation and the 
prime number theorem to each sum.)

We bound $|\mathscr{P}_2(X,Y)|$ and $|\mathscr{P}_3(Y)|$
similarly, but we need the following non-trivial bound
\cite[Lemma 2.5]{C2003}:
\begin{align*}
|S_j(X;D,k)|
\ll 
\br{\frac{\sqrt{X}}{k} + 1}
\frac{\sqrt{X}\log\log X}{k\sqrt{D}
\log\br{\frac{\sqrt{X}-1}{k}}},
\end{align*}  
provided $k < \sqrt{X} - 1$. Thus,
\begin{align}\label{(3.19)}
\begin{split} 
 |\mathscr{P}_2(X,Y)|  
 & \le \sum_{Y < q \le \log X} |S_j(X;D,q)| \\
 & \ll \sum_{Y < q \le \log X} 
        \br{\frac{\sqrt{X}}{q} + 1}
            \frac{\sqrt{X}\log\log X}{q\sqrt{D}
            \log\br{\frac{\sqrt{X}-1}{q}}} \\
 & \ll \frac{X\log\log X}{\log\br{\frac{\sqrt{X}-1}{\log X}}}
       \sum_{q > Y}\frac{1}{q^2}
     + \frac{\sqrt{X}\log\log X}
            {\log\br{\frac{\sqrt{X}-1}{\log X}}}
       \sum_{q \le \log X}\frac{1}{q} \\
 & \ll \frac{X\log\log X}{(\log X)Y\log Y},
\end{split} 
\end{align}
and
\begin{align}\label{(3.20)}
\begin{split}
|\mathscr{P}_3(Y)| 
 & \le \sum_{Y < k \le Y^2} |S_i(X;D,k)| \\
 & \ll \sum_{Y \le k \le Y^2} 
        \br{\frac{\sqrt{X}}{k} + 1}
        \frac{\sqrt{X}\log\log X}
             {k\sqrt{D}\log\br{\frac{\sqrt{X}-1}{k}}} \\
 & \ll \frac{X\log\log X}{\log\br{\frac{\sqrt{X}-1}{Y}}}
        \sum_{Y \le k \le Y^2}\frac{1}{k^2}
     + \frac{\sqrt{X}\log\log X}{\log\br{\frac{\sqrt{X}-1}{Y}}}
        \sum_{Y \le k \le Y^2}\frac{1}{k} \\
 & \ll \frac{X\log\log X}{(\log X)Y} 
     + \frac{\sqrt{X}(\log\log X)(\log Y)}{\log X}.
 \end{split}
\end{align}
Since $Y \le \log X$, putting \eqref{(3.16)} --- \eqref{(3.20)}
into \eqref{(3.15)}, we obtain \eqref{(3.14)}.
\end{proof}

\section{Proof of Theorem 1.1}\label{Section 4}

Let $X \ge 2$ and set 
\begin{align*}
 Y = Y(X) \defeq \frac{X^{1/5}}{(\log X)^{2/5}}.
\end{align*}
We proceed on GRH, so that partial summation applied to 
\eqref{(3.2)} gives
\begin{align}\label{(4.1)}
 \begin{split}
 \sums{p \le X}{\text{$p \nmid N$, $k \mid d_p$}} p 
& = X \sums{p \le X}{\text{$p \nmid N$, $k \mid d_p$}} 1 
    - \int_{2}^{X} \Big(\textstyle
   \sums{p \le t}{\text{$p \nmid N$,$k \mid d_p$}} 1\Big)\dd{t} \\
& = \frac{X\li{X}}{\deg{\L{k}}}  
    - \frac{1}{\deg{\L{k}}}\int_{2}^{X} \li{t} \dd{t}   
    + \BigO{ X^{3/2}\log(XN)} \\ 
& = \frac{\li{X^2}}{\deg{\L{k}}} + \BigO{ X^{3/2}\log(XN)}.
 \end{split}
\end{align}
Here we have used that 
$\int_2^X \li{t} \dd{t} = X\li{X} - \li{X^2} + \BigO{1}$. 

Recall that
$\expep = |\Ep(\fp)|/d_p = (p + 1 - a_p)/d_p$
if $p \nmid N$, and by definition $\expep = 0$ otherwise, hence
\begin{align}\label{(4.2)}
 \sum_{p \le X} \expep 
 = \sums{p \le X}{p \nmid N} \frac{p}{d_p} 
   + \sums{p \le X}{p \nmid N} \frac{1 - a_p}{d_p} 
       \eqdef \mathcal{T}_0 + \mathcal{E}_0.
\end{align}
Since $\ab{a_p} \le 2\sqrt{p}$ by Hasse's inequality, we have
\begin{align}\label{(4.3)}
\mathcal{E}_0 
  \defeq \sums{p \le X}{p \nmid N} \frac{1 - a_p}{d_p} 
    \ll \sum_{p \le X} \sqrt{p} 
      \le X^{1/2}\sum_{p \le X} 1 
        \le X^{3/2}.
\end{align}
We write
\begin{align*}
\mathcal{T}_0 
  \defeq \sums{p \le X}{p \nmid N} \frac{p}{d_p} 
  = \sums{p \le X}{\text{$p \nmid N$, $d_p \le Y$}} \frac{p}{d_p} 
    + \sums{p \le X}{\text{$p \nmid N$, $d_p > Y$}} \frac{p}{d_p} 
        \eqdef \mathcal{T}_{1} + \mathcal{E}_{1}.
\end{align*}
To bound $\mathcal{E}_1$, we apply \eqref{(3.13)}, and then use 
the definition of $Y$:
\begin{align}\label{(4.4)}
\mathcal{E}_{1} 
 \defeq 
   \sums{p \le X}{\text{$p \nmid N$, $d_p > Y$}} \frac{p}{d_p} 
     \le \frac{X}{Y} 
           \sums{p \le X}{\text{$p \nmid N$, $d_p > Y$}} 1 
       \ll \frac{X^{5/2}}{Y^3} + \frac{X^{3/2}\log X}{Y} 
         \ll X^{19/10}(\log X)^{6/5}.
\end{align}
We partition $\mathcal{T}_1$, using the identity 
$\frac{1}{k} = \sum_{hj \mid k} \frac{\mu(h)}{j}$, as follows:
\begin{align*}
\mathcal{T}_{1} 
 \defeq \sums{p \le X}{\text{$p \nmid N$, $d_p \le Y$}} p 
          \sum_{hj \mid d_p} \frac{\mu(h)}{j} 
 = \sums{p \le X}{p \nmid N} p 
     \sums{hj \mid d_p}{hj \le Y} \frac{\mu(h)}{j} 
    - \sums{p \le X}{\text{$p \nmid N$, $d_p > Y$}} p 
       \sums{hj \mid d_p}{hj \le Y} \frac{\mu(h)}{j} 
 \eqdef \mathcal{T}_{11} - \mathcal{E}_{11}.
\end{align*}

We now consider $\mathcal{E}_{11}$, making yet another partition:
\begin{align*}
 \mathcal{E}_{11} 
  \defeq  \sumss{p \le X}{p \nmid N}{d_p > Y^{2}} p 
           \sums{hj \mid d_p}{hj \le Y} \frac{\mu(h)}{j} 
         + \sumss{p \le X}{p \nmid N}{Y < d_p \le Y^{2}} p 
            \sums{hj \mid d_p}{hj \le Y} \frac{\mu(h)}{j} 
 \eqdef \mathcal{E}_{12} + \mathcal{E}_{13}.
\end{align*}
Next, we note that
\begin{align*}
 \sums{hj \mid d_p}{hj \le Y} \frac{\mu(h)}{j} 
  \ll \sum_{j \le Y} \frac{1}{j}  \sum_{h \le Y/j} 1 
    \le Y\sum_{j \le Y} \frac{1}{j^2} 
      \ll Y.
\end{align*}
Thus, by \eqref{(3.13)} (with $Y^{2}$ in place of $Y$), and by 
the definition of $Y$, we have
\begin{align}\label{(4.5)}
  \begin{split}
 \mathcal{E}_{12}   
   & \defeq \sumss{p \le X}{p \nmid N}{d_p > Y^{2}} p 
            \sums{hj \mid d_p}{hj \le Y} \frac{\mu(h)}{j} 
       \ll  XY \sumss{p \le X}{p \nmid N}{d_p > Y^{2}} 1 \\
   & \hspace{140pt} 
       \ll \frac{X^{5/2}}{Y^{3}} + X^{3/2}Y\log X 
        \ll X^{19/10}(\log X)^{6/5}.
 \end{split}          
\end{align}

For $\mathcal{E}_{13}$, we use
\begin{align*}
 \sums{hj \mid d_p}{hj \le Y} \frac{\mu(h)}{j} 
  \ll \sum_{h \mid d_p} 1  \sum_{j \mid d_p}  \frac{1}{j} 
 = \frac{\tau(d_p)}{d_p} \sum_{j \mid d_p} \frac{d_p}{j} 
 = \frac{\tau(d_p)\sigma(d_p)}{d_p}.
\end{align*}
Thus,
\begin{align*}
\mathcal{E}_{13}  
& \defeq 
 \sumss{p \le X}{p \nmid N}{Y < d_p \le Y^{2}} p 
   \sums{hj \mid d_p}{hj \le Y} \frac{\mu(h)}{j} \\
 & \hspace{50pt}
  \ll  \sumss{p \le X}{p \nmid N}{Y < d_p \le Y^{2}} p 
                 \cdot \frac{\tau(d_p)\sigma(d_p)}{d_p} 
   \le \sum_{Y < k \le Y^{2}} \frac{\tau(k)\sigma(k)}{k} 
        \sums{p \le X}{\text{$p \nmid N$, $k \mid d_p$}} p.
\end{align*}
We apply \eqref{(4.1)} to the last sum, noting that 
$\li{X^2} \ll X^2/\log X$. Then we use \eqref{(3.8)}, as well as 
the bound
$\sum_{k \le Y^{2}} \frac{\tau(k)\sigma(k)}{k} 
   \ll \int_2^{Y^{2}} \log t \dd{t} \ll Y^{2} \log Y$.
(Apply partial summation to 
$\sum_{k \le t} \tau(k)\sigma(k) \ll t^2\log t$.) Thus,
\begin{align*}
 & \sum_{Y < k \le Y^{2}} \frac{\tau(k)\sigma(k)}{k} 
 \sums{p \le X}{\text{$p \nmid N$, $k \mid d_p$}} p \\
 & \hspace{50pt}  
    \ll \frac{X^2}{\log X} \sum_{k > Y} 
        \frac{\tau(k)\sigma(k)}{k\deg{\L{k}}} 
      + X^{3/2} \log(XN) 
         \sum_{k \le Y^{2}} \frac{\tau(k)\sigma(k)}{k} \\
 & \hspace{50pt} 
    \ll 
   \begin{cases}
     \frac{X^2 \log Y}{Y\log X} + X^{3/2}(\log(XN))Y^{2}\log Y
    & \text{if $E$ has CM,} \\
     \frac{X^2}{\log X}\cdot \frac{B_E \log Y}{Y^3} 
        + X^{3/2}(\log(XN))Y^{2}\log Y.
    & \text{if $E$ is a non-CM curve.}
   \end{cases}
\end{align*}
Combining and using the definition of $Y$, we obtain
\begin{align}\label{(4.6)}
 \mathcal{E}_{13} 
  \ll  
   \begin{cases}
     X^{19/10}(\log(XN))^{6/5} 
    & \text{if $E$ has CM,} \\
     B_EX^{7/5}(\log X)^{6/5}
       + X^{19/10}(\log(XN))^{6/5}
    & \text{if $E$ is a non-CM curve.}
   \end{cases}
\end{align}

Finally we consider $\mathcal{T}_{11}$. By \eqref{(4.1)}, and 
since 
\begin{align*}
 \abs{\sum_{hj \le Y} \frac{\mu(h)}{j}} 
  = \abs{\sum_{k \le Y} \sum_{hj = k}\frac{\mu{(h)}}{j} }  
     \le \sum_{k \le Y}\abs{ \sum_{hj = k}\frac{\mu{(h)}}{j}} 
      \le \sum_{k \le Y} \frac{\phi(k)}{k} \le Y,
\end{align*}
we have
\begin{align}\label{(4.7)}
\begin{split}
\mathcal{T}_{11} 
& \defeq 
 \sums{p \le X}{p \nmid N} p 
  \sums{hj \mid d_p}{hj \le Y} \frac{\mu(h)}{j} 
= \sum_{hj \le Y} \frac{\mu(h)}{j} 
   \sums{p \le X}{\text{$p \nmid N$, $hj \mid d_p$}} p \\ 
& \phantom{:}= \li{X^2} \sum_{hj \le Y} \frac{\mu(h)}{j} 
            \cdot \frac{1}{\deg{\L{hj}}}  
     + \BigO{X^{3/2}Y \log(XN)}.
\end{split}
\end{align}
Now, for prime powers $q^m$ we have
\begin{align*}
 \sum_{hj = q^m} \frac{\mu(h)}{j} 
   = \frac{1}{q^m} - \frac{1}{q^{m-1}} = \frac{1-q}{q^m},
\end{align*}
and so by multiplicativity we have
\begin{align*}
\sum_{hj = k} \frac{\mu(h)}{j} 
= \prod_{q^m \| k}\frac{1-q}{q^m}
= \frac{(-1)^{\omega(k)}}{k}\prod_{q \mid k}(q-1)
= \frac{(-1)^{\omega(k)}\phi(\rad(k))}{k}.
\end{align*}
(Here, $q^m \| k$ means that $q^m \mid k$ but $q^{m+1} \nmid k$.) 
Therefore, setting $c_E$ as in \eqref{(1.1)}, and using 
\eqref{(3.7)}, we obtain
\begin{align*}
 \sum_{hj \le Y} \frac{\mu(h)}{j}\cdot \frac{1}{\deg{\L{hj}}} 
 & = \sum_{k = 1}^{\infty} 
      \frac{(-1)^{\omega(k)}\phi(\rad(k))}{k\deg{\L{k}}} 
    + \BigO{\sum_{k > Y} \frac{1}{\deg{\L{k}}}} \\
 & = 
  c_E +
   \begin{cases}
     \BigO{1/Y} 
        &  \text{if $E$ has CM,} \\
     \BigO{B_E/Y^3}  
        &  \text{if $E$ is a non-CM curve.}
   \end{cases}
\end{align*}
Putting this into \eqref{(4.7)}, then using the definition of $Y$, 
we obtain
\begin{align}\label{(4.8)}
\begin{split}
 \mathcal{T}_{11} 
& = c_E \cdot \li{X^2} \\
& \hspace{30pt} + 
  \begin{cases}
    \BigO{\frac{X^{9/5}}{(\log X)^{3/5}}}  
      + \BigO{X^{17/10}(\log (XN))^{3/5}} 
       & \text{if $E$ has CM,} \\
    \BigO{B_EX^{7/5}(\log X)^{1/5}}
      + \BigO{X^{17/10}(\log (XN))^{3/5}}   
       & \text{if $E$ is a non-CM curve.}
  \end{cases}
\end{split}
\end{align}

Gathering the estimates \eqref{(4.8)}, \eqref{(4.6)}, 
\eqref{(4.5)}, \eqref{(4.4)}, \eqref{(4.3)}, and 
\eqref{(4.2)}, the largest error term being of size
$\BigO{X^{19/10}(\log(XN))^{6/5}}$, we obtain
\begin{align*}
 \sum_{p \le X} \expep 
  = \mathcal{T}_{11} 
      - (\mathcal{E}_{12} 
         + \mathcal{E}_{13}) 
      + \mathcal{E}_1 
      + \mathcal{E}_0 
  = c_E \cdot \li{X^2}  + \BigO{X^{19/10}(\log(XN))^{6/5}},
\end{align*}
plus $\BigO{B_EX^{7/5}(\log X)^{6/5}}$ if $E$ is a non-CM curve, 
this term being dominated by the other $\Oh$-term once 
$X \ge B_E^2$. \qed 

\section{Proof of Theorem 1.2}\label{Section 5}

We now fix an elliptic curve $E$ defined over $\Q$, of conductor 
$N$. We suppose that $E$ has complex multiplication by an order in 
the imaginary quadratic field $K = \Q(\sqrt{-D})$. We know that 
there are only nine possibilities for $D$, namely 
$D \in \{1,2,3,7,11,19,43,67,163\}$. (See 
\cite[Appendix C, \S 11]{S1986}.) 

Let $c_1$ be the absolute positive constant of 
Lemma \ref{Lemma 3.3}(a), and let $c_0, c_2$ be positive 
constants to be specified presently. Suppose 
$X \ge \exp\br{c_0N^2}$ and set
\begin{align*}
Y = Y(X,N) 
  \defeq c_2\br{\frac{\log X}{N^2}}^{\frac{1}{24}}.
\end{align*}
We choose $c_0$ and $c_2$ so that $2 < Y \le (\log X)^{1/3}$ and
$2c_1Y^{24}N^2 \le \log X$.

Similarly to \eqref{(4.1)},
partial summation and Lemma \ref{Lemma 3.3}(a)(ii) give
\begin{align}\label{(5.1)}
\sums{p \le X}{\text{$p \nmid N$, $k \mid d_p$}} p
= \frac{\li{X^2}}{\deg{\L{k}}} 
+ \BigO{X^2\exp\br{-B(\log X)^{3/8}}},
\end{align}
provided $c_1k^{8}N^2 \le \log X$. One of the error terms 
involved is
\begin{align*}
\int_2^X 
  \Big(\textstyle
        \sums{p \le t}{\text{$p \nmid N$, $k \mid d_p$}} 1 
       - \frac{\li{t}}{\deg{\L{k}}} \Big) \dd{t}.
\end{align*}
We can apply \eqref{(5.1)} to 
$\int_{\exp(c_1k^{8}N^2)}^X(\cdots) \dd{t}$,
but we can only apply a trivial bound to the rest of the integral:
\begin{align*}
\int_2^{\exp(c_1k^{8}N^2)}
  \Big(\textstyle
        \sums{p \le t}{\text{$p \nmid N$, $k \mid d_p$}} 1 
       - \frac{\li{t}}{\deg{\L{k}}} \Big) \dd{t} 
  \ll \displaystyle\int_2^{\exp(c_1k^{8}N^2)} t \dd{t} 
   \ll \exp\br{2c_1k^{8}N^2}.
\end{align*}
This is $\BigO{X}$ if $k \le Y^3$, because 
$2c_1Y^{24}N^2 \le \log X$. 

We now proceed as in the proof of Theorem \ref{Theorem 1.1}. 
The differences are as follows. Analogous with 
\eqref{(4.4)} and \eqref{(4.5)} we have, 
by \eqref{(3.14)}, that
\begin{align}\label{(5.2)}
\mathcal{E}_{1} 
\defeq \sums{p \le X}{\text{$p \nmid N$, $d_p > Y$}} \frac{p}{d_p} 
\le \frac{X}{Y} \sums{p \le X}{\text{$p \nmid N$, $d_p > Y$}} 1 
\ll \frac{X^2}{\log X}\cdot \frac{\log\log X}{Y^2},
\end{align}
and
\begin{align}\label{(5.3)}
\mathcal{E}_{12}  
\defeq  
\sumss{p \le X}{p \nmid N}{d_p > Y^{3}} p 
\sums{hj \mid d_p}{hj \le Y} \frac{\mu(h)}{j} 
\ll  
XY \sumss{p \le X}{p \nmid N}{d_p > Y^{3}} 1 
\ll 
\frac{X^2}{\log X}\cdot \frac{\log\log X}{Y^2}.
\end{align}

Analogous with \eqref{(4.6)} we have
\begin{align*}
\mathcal{E}_{13} 
\defeq 
\sumss{p \le X}{p \nmid N}{Y < d_p \le Y^{3}} p 
\sums{hj \mid d_p}{hj \le Y} \frac{\mu(h)}{j} 
\le 
\sum_{Y < k \le Y^{3}} \frac{\tau(k)\sigma(k)}{k} 
\sums{p \le X}{\text{$p \nmid N$, $k \mid d_p$}} p.
\end{align*}
Since \eqref{(5.1)} holds uniformly for $k \le Y^3$, we may apply 
it to the last sum to obtain
\begin{align*}
\mathcal{E}_{13}       
  & \ll 
     \frac{X^2}{\log X} 
      \sum_{k > Y} \frac{\tau(k)\sigma(k)}{k\deg{\L{k}}} 
     + X^2\exp\br{-B(\log X)^{3/8}} 
        \sum_{k \le Y^{3}} \frac{\tau(k)\sigma(k)}{k}.
\end{align*}
We use \eqref{(3.8)} to bound the second last sum 
and an elementary bound for the last sum. Thus, we have
\begin{align}\label{(5.4)}
\mathcal{E}_{13} 
\ll 
  \frac{X^2}{\log X}\cdot \frac{\log Y}{Y} 
   + X^2\exp\br{-B(\log X)^{3/8}}\cdot Y^{3}\log Y
\ll 
  \frac{X^2}{\log X}\cdot \frac{\log Y}{Y}.
\end{align}

Analogous with \eqref{(4.8)}, we have, by \eqref{(5.1)},
\begin{align}\label{(5.5)}
\begin{split}
 \mathcal{T}_{11} 
 & \defeq 
    \sums{p \le X}{p \nmid N} p 
     \sums{hj \mid d_p}{hj \le Y} \frac{\mu(h)}{j} 
 = 
   \li{X^2} \br{c_E +  \BigO{\frac{1}{Y}}} 
    + YX^2\exp\br{-B(\log X)^{3/8}} \\
& \phantom{:}= 
  c_E \cdot \li{X^2} + \BigO{\frac{X^2}{Y\log X}}.
\end{split}
\end{align}
Gathering the estimates \eqref{(5.5)}, \eqref{(5.4)}, 
\eqref{(5.3)}, \eqref{(5.2)}, and \eqref{(4.2)},  we obtain
\begin{align*}
 \sum_{p \le X} \expep 
  & = 
   \mathcal{T}_{11} 
  - (\mathcal{E}_{12} 
      + \mathcal{E}_{13}) 
  + \mathcal{E}_1 
  + \mathcal{E}_0 \\
  & = 
 c_E \cdot \li{X^2}  
    + \BigO{\frac{X^2}{\log X} \br{\frac{\log\log X}{Y^2} 
      + \frac{\log Y}{Y}} }.
\end{align*}
Since $Y = (N^{-2}\log X)^{1/24}$, the theorem follows. \qed

\section{Proof of Theorem 1.3}\label{Section 6}

Let $E$ be an elliptic curve defined over $\Q$, of conductor $N$. 
We make no assumptions as to whether or not $E$ has CM.
Let $c_1$ be the absolute positive constant of 
Lemma \ref{Lemma 3.3}(a), and let $c_0, c_2$ be positive constants 
to be specified presently. Suppose $X \ge \exp\br{c_0N^2}$ and set
\begin{align*}
Y = Y(X,N) 
  \defeq c_2\log\left[ \frac{\log X}{N^2} \right]^{\frac{1}{42}}.
\end{align*}
We choose $c_0$ and $c_2$ so that 
$2c_1\exp\br{42Y}N^2 \le \log X$.

As in the proofs of the first two theorems, we have
\begin{align}\label{(6.1)}
\sum_{p \le X} \expep 
 =  \sums{p \le X}{\text{$p \nmid N$, $d_p \le Y$}} \frac{p}{d_p} 
   + \sums{p \le X}{\text{$p \nmid N$, $d_p > Y$}} \frac{p}{d_p}
     + \sums{p \le X}{p \nmid N} \frac{1 - a_p}{d_p} 
       \eqdef \mathcal{T}_1 + \mathcal{E}_{1} + \BigO{X^{3/2}}.
\end{align}
Since $d_p \le 2\sqrt{p}$, and since $p \equiv 1 \bmod j$ if 
$d_p = j$ by Proposition \ref{Proposition 3.2}(a), we have
\begin{align*}
  \sums{p \le X}{\text{$p \nmid N$, $d_p > Y$}} \frac{1}{d_p}
     = \sum_{Y < j \le 2\sqrt{X}} \frac{1}{k} 
        \sums{p \le X}{\text{$p \nmid N$, $d_p = j$}} 1
     \le \sum_{Y < j \le 2\sqrt{X}} \frac{1}{j} 
           \sums{p \le X}{p \equiv 1 \bmod j} 1. 
\end{align*}     
By the Brun-Titchmarsh inequality \cite[Theorem 3.7]{HR1974},
\begin{align*}
\sum_{Y < j \le 2\sqrt{X}} \frac{1}{j} 
  \sums{p \le X}{p \equiv 1 \bmod j} 1 
 \ll \sum_{Y < j \le 2\sqrt{X}} 
       \frac{1}{j}\cdot \frac{X}{\phi(j)\log(X/j)} 
 \ll \frac{X}{\log X}\sum_{j > Y} \frac{1}{j\phi(j)} 
 \ll \frac{X}{Y\log X}
\end{align*}
(cf.~\eqref{(3.9)} for the last bound). Hence
\begin{align}\label{(6.2)}
\mathcal{E}_1 
  \defeq 
    \sums{p \le X}{\text{$p \nmid N$, $d_p > Y$}} \frac{p}{d_p}
  \le 
    X \sums{p \le X}{\text{$p \nmid N$, $d_p > Y$}} \frac{1}{d_p}
  \ll \frac{X^2}{Y\log X}.
\end{align}

We claim that, with $c_E$ as in \eqref{(1.1)} and $B_E$ as in 
Proposition \ref{Proposition 3.2}(c), we have
\begin{align}\label{(6.3)}
\mathcal{T}_1 
  & \le 
      \begin{cases}
      c_E \cdot \li{X^2} + \BigO{\frac{X^2}{Y\log X}} 
        & \text{if $E$ has CM,} \\
      c_E \cdot \li{X^2} + \BigO{\frac{X^2}{Y\log X} 
           + \frac{B_EX^2}{Y^3\log X}} 
        & \text{if $E$ is a non-CM curve.}
      \end{cases} 
\end{align}
Here and in the next two instances, by $F \le G + \BigO{H}$, we 
mean that either $F - G < 0$ or $0 \le F - G \ll H$. The theorem 
follows by combining \eqref{(6.3)} and \eqref{(6.2)} with 
\eqref{(6.1)}.

Let us prove our claim. By an inclusion-exclusion argument, the 
following inequality holds for any number $X \ge 2$ and any 
integer $Q$:
\begin{align*}
\sums{p \le X}{\text{$p \nmid N$, $d_p = j$}} p
  \le \sum_{h \mid Q} \mu(h) 
        \sums{p \le X}{\text{$p \nmid N$, $hj \mid d_p$}} p.
\end{align*}
Thus, applying partial summation and \eqref{(3.1)} to the last 
sum (as we did to obtain \eqref{(5.1)}), and noting that 
$\sum_{h \mid Q} |\mu(h)| = 2^{\omega(Q)}$, we obtain
\begin{align*}
\sums{p \le X}{\text{$p \nmid N$, $d_p = j$}} p 
  \le \li{X^2} \sum_{h \mid Q} \frac{\mu(h)}{\deg{\L{hj}}} 
     + \BigO{2^{\omega(Q)}X^2\exp\br{-B(\log X)^{5/14}}}, 
\end{align*}
provided $2c_1(hj)^{14}N^2 \le \log X$ for every $h \mid Q$.

We set $Q = Q(Y) \defeq \prod_{q \le Y} q$. Then $\log Q \sim Y$ 
as $Y \to \infty$ by the prime number theorem, and we may suppose 
that $X$ and $Y$ are large enough so that $QY \le \exp\br{3Y}$,
that is $2c_1(QY)^{14}N^2 \le 2c_1\exp\br{42Y}N^2 \le \log X$ (by
definition of $Y$ and our choice of constants). 

Thus, since $2^{\omega(Q)} = 2^Y \ll (\log X)^{1/42}$, and since
$\sum_{j \le Y} \frac{1}{j} \ll \log Y \ll \log\log\log X$, we 
have
\begin{align}\label{(6.4)}
\begin{split}
\mathcal{T}_1 
  & \defeq 
    \sums{p \le X}{\text{$p \nmid N$, $d_p \le Y$}} \frac{p}{d_p} 
    = \sum_{j \le Y} \frac{1}{j} 
      \sums{p \le X}{\text{$p \nmid N$, $d_p = j$ }} p \\
   & \hspace{120pt} 
    \le \li{X^2} \sum_{j \le Y} 
      \sum_{h \mid Q} \frac{\mu(h)}{j\deg{\L{hj}}}
      + \BigO{X^2\exp\br{-\textstyle\frac{1}{2}B(\log X)^{5/14}}}.
\end{split}
\end{align}
Letting $S = \{\text{$hj$ : $h \mid Q$ and $j \le Y$}\}$, we have 
\begin{align*}
 \sum_{j \le Y} \sum_{h \mid Q} \frac{\mu(h)}{j\deg{\L{hj}}}
 =
  \sum_{k \in S} \frac{1}{\deg{\L{k}}}
   \sum_{hj = k} \frac{\mu(h)}{j}
 =
 \sum_{k \in S} \frac{1}{k\deg{\L{k}}}
  \sum_{h \mid k} \mu(h)\cdot h.
\end{align*}
We complete the sum over $k$, noting that $k \not\in S$ implies 
either $k > Y$ or $k = hj$ with $\mu(h) = 0$, and that
$\sum_{h \mid k} \mu(h)\cdot h = (-1)^{\omega(k)}\phi(\rad(k))$,
obtaining
\begin{align*}
\sum_{j \le Y} \sum_{h \mid Q} \frac{\mu(h)}{j\deg{\L{hj}}}
 & = \sum_{k=1}^{\infty}  
      \frac{(-1)^{\omega(k)}\phi(\rad(k))}{k\deg{\L{k}}}                   
     + \BigO{\sum_{k > Y} \frac{1}{\deg{\L{k}}}} \\
 & = 
 c_E + 
\begin{cases} 
     \BigO{1/Y} 
       & \text{if $E$ has CM,} \\
     \BigO{B_E/Y^3} 
       & \text{if $E$ is a non-CM curve.}
     \end{cases}
\end{align*} 
by \eqref{(3.7)}. Combining this with \eqref{(6.4)} gives 
\eqref{(6.3)}. \qed

\section{Further remarks on the constant $c_E$}\label{Section 7} 

Notation in this section is as in the proof of Proposition
\ref{Proposition 3.2}. Let us first assume that $E$ is a non-CM
curve. Let 
$G = \varprojlim \mathrm{GL}_{2}(\Z/n\Z) 
      \cong \prod_{\ell}\mathrm{GL}_{2} (\Z_{\ell})$, 
the product being over all primes $\ell$, and let $H$ denote the 
image of $\Gal{\QQ}$ in $G$ under $\rho_E$. 
For $n \ge 1$ an integer, there is a natural 
projection map from $G$ to $\mathrm{GL}_{2}(\Z/n\Z)$; let
\begin{align*}
\Gamma_n \defeq \ker \br{G \to \mathrm{GL}_{2}(\Z/n\Z)}.
\end{align*}
By Serre's open image theorem, the index $(G:H)$ is finite, hence
there exist $n$ such that $\Gamma_n < H$; let $m$ denote the 
smallest such $n$. (In order to show that the growth of 
$\deg{\L{k}}$ is ``regular'', it is convenient to work with a large 
subgroup $\prod_\ell K_{\ell} \subset H$, with each $K_{\ell}$ 
having  simple structure.)

Now, the image of $\Gal{\QQ}$ in 
$\mathrm{GL}_{2}(\Z/k\Z)$ can be obtained by composing the maps
\begin{align*}
H \hookrightarrow G \twoheadrightarrow G/\Gamma_k \cong
\mathrm{GL}_{2}(\Z/k\Z),
\end{align*}
and hence
\begin{align*}
\deg{\L{k}} = [\L{k} : \Q] = | H / H \cap \Gamma_k|.
\end{align*}

Write $m = \prod_{p \mid m} p^{m_{p}}$ and let 
$k = \prod_{p \mid k} p^{k_{p}}$ be given.  

\noindent {\em Claim:}  If $k_{p} \ge m_{p}$ for some $p$ and 
$a \ge 1$, then
\begin{align*}
|H/H \cap \Gamma_{p^{a}\cdot k}|
= |H/H \cap \Gamma_{k}| \cdot 
|\Gamma_{p^{k_p}}  / \Gamma_{p^{a+k_p }} |.
\end{align*}
Moreover, if $k_{p}=0$, we have
$\Gamma_{p^{k_p}}  / \Gamma_{p^{a+k_p }} = \Gamma_1/\Gamma_{p^a}
\cong \mathrm{GL}_{2}(\Z/p^{a}\Z)$, and if $k_{p}>0$ then
\begin{align*}
|\Gamma_{p^{k_p}}  / \Gamma_{p^{a+k_p }} | = p^{4a}.
\end{align*}

\noindent {\em Proof of claim:}
We note that
\begin{align*}
|H/H \cap \Gamma_{p^{a}\cdot k}| =
|H/H \cap \Gamma_{ k}| \cdot
|(H \cap \Gamma_k) / (H \cap \Gamma_{p^{a}\cdot k})| 
\end{align*}
and
\begin{align}\label{(7.1)}
|(H \cap \Gamma_k) / (H \cap \Gamma_{p^{a}\cdot k})| 
=
\frac{ |(H \cap \Gamma_k)/(\Gamma_m \cap \Gamma_k) | 
\cdot
| (\Gamma_m \cap \Gamma_k)/(\Gamma_m \cap \Gamma_{p^{a}k}) | }
{ | (H \cap \Gamma_{p^{a}k})/(\Gamma_m \cap \Gamma_{p^{a}k}) | }.
\end{align}
Let $N = \Gamma_m \cap \Gamma_k$ and $S = H \cap \Gamma_{p^ak}$.
Since $\Gamma_m < H$ and (trivially) $\Gamma_{p^ak} < \Gamma_k$, 
we find that
\begin{align*}
S \cap N 
= (H \cap \Gamma_{p^{a}k}) \cap(\Gamma_m \cap \Gamma_k) 
= \Gamma_m \cap \Gamma_{p^ak}.
\end{align*}
Moreover, since $k_{p}\ge m_{p}$, given any 
$h_1 \in H \cap \Gamma_k$ we can find 
$h_2 \in H \cap \Gamma_{p^ak}$ such that 
$h_1 = h_2 \cdot (1,1,\ldots,1,\gamma_p,1,1,\ldots)$, where 
$\gamma_p = I + p^{k_{p}}M$ and $M \in \mathrm{Mat}_{2}(\Z_p)$, 
and consequently
\begin{align*}
S \cdot N 
= (H \cap \Gamma_{p^ak}) \cdot (\Gamma_m \cap \Gamma_k)
= H \cap \Gamma_k.
\end{align*}
Now, by the second isomorphism theorem of group theory,
$SN/N \cong S/S\cap N$, hence
\begin{align*}
(H \cap \Gamma_k)/(\Gamma_m \cap \Gamma_k)   
= SN/N 
\cong  
S/S\cap N 
= (H \cap \Gamma_{p^ak}) / (\Gamma_m \cap \Gamma_{p^ak}),
\end{align*}
which, together with \eqref{(7.1)}, implies that
\begin{align*}
|(H \cap \Gamma_k) / (H \cap \Gamma_{p^{a}\cdot k})| =
|(\Gamma_m \cap \Gamma_k)/(\Gamma_m \cap \Gamma_{p^{a}k})|.
\end{align*}

With $[a,b]$ denoting the least common multiple of two integers 
$a,b$, we have $\Gamma_a \cap \Gamma_b = \Gamma_{[a,b]}$. Further, 
if $a\mid b$ and $(ab,c)=1$, then 
$\Gamma_{ac}/\Gamma_{bc} \cong \Gamma_a/\Gamma_b$. Thus, again 
using that $k_{p} \ge m_{p}$, we find that 
\begin{align*}
(\Gamma_m \cap \Gamma_k)/(\Gamma_m \cap \Gamma_{p^{a}k})
=
\Gamma_{[m,k]}/\Gamma_{[m,p^{a}k]}
\cong
\Gamma_{p^{k_p}} / \Gamma_{p^{k_p+a}}.
\end{align*}
Hence
\begin{align*}
|H/H \cap \Gamma_{p^{a}\cdot k}|
=
|H/H \cap \Gamma_{ k}| \cdot
|(H \cap \Gamma_k) / (H \cap \Gamma_{p^{a}\cdot k})| 
=
|H/H \cap \Gamma_{ k}| \cdot
|\Gamma_{p^{k_p}} / \Gamma_{p^{k_p+a}}|,
\end{align*}
and the first part of the claim is proved. The latter part of the 
claim follows from
\begin{align*}
|\Gamma_{p^{k_p}}  / \Gamma_{p^{a+k_p }}|
= \prod_{j=k_{p}}^{a+k_{p}-1} 
|\Gamma_{p^{j}}  / \Gamma_{p^{j+1}}|
\end{align*}
together with 
$\Gamma_{p^{j}}  / \Gamma_{p^{j+1}} 
\cong \Gamma_{p}  /  \Gamma_{p^2}$, and 
$|\Gamma_p/\Gamma_{p^2}| = p^{4}$.

Now, let $\langle m \rangle$ be the set containing $1$ and the 
positive integers composed only of primes dividing $m$. For any 
$k$ we have $k = hj$ with $h \in \langle m \rangle$ and 
$(j,m) = 1$. We can further write $h = h_1h_2$ in a unique way with 
$h_1 \mid m$, $(h_1,h_2) = 1$, and $\nu_p > m_p$ for every 
$p \mid h_2$, where $p^{\nu_p} \| h_2$ and $p^{m_p} \| m$. 
From the above claim it follows that 
$\deg{\L{k}} = \deg{\L{h}}\cdot \deg{\L{j}}$, that 
$\deg{\L{j}} = |\mathrm{GL}_2(\Z/j\Z)|$, and that 
\begin{align*}
  \deg{\L{h}} = r_E(h_1) 
  \prod_{p \mid h_2} p^{4(\nu_{p} - m_{p})},
\end{align*}
where $r_E(h_1)$ is a rational number depending only on $E$ and 
$h_1$, $p^{\nu_{p}} \| h_2$, and $p^{m_{p}} \| m$.

We therefore have 
\begin{align*}
c_E & = 
 \sum_{h \in \langle m \rangle} 
    \frac{(-1)^{\omega(h)} \phi(\rad(h))}{h\deg{\L{h}}}
  \prod_{q \nmid m} 
   \br{1 - \sum_{k \ge 1} \frac{q-1}{q^k\deg{\L{q^k}}}} \\ 
 & = c\cdot 
  \prod_{q \mid m} \br{1 - \frac{q^3}{(q^2-1)(q^5-1)}}^{-1}
  \cdot \sum_{h \in \langle m \rangle} 
    \frac{(-1)^{\omega(h)} \phi(\rad(h))}{h\deg{\L{h}}},
\end{align*}
with $c$ as defined in \eqref{(1.2)}. 
Using the fact that 
$\deg{\L{h}} = r_E(h_1) 
  \prod_{p \mid h_2} p^{4(\nu_{p} - m_{p})}$, it is a 
straightforward matter to show that the sum over 
$h \in \langle m \rangle$ is equal to the rational number
\begin{align*}
 \prod_{q \mid m} 
        \br{1 - \frac{q-1}{q^{m_q}(q^5 - 1)}}
 \sum_{h_1 \mid m} 
 \frac{(-1)^{\omega(h_1)} \phi(\rad(h_1))}{h_1r_E(h_1)}
 \prod_{q \mid h_1}
        \br{1 - \frac{q-1}{q^{m_q}(q^5 - 1)}}^{-1}.
\end{align*}

The CM case is similar, except that $\deg{\L{q^{k}}}$ (for all but
finitely many $q$ and $k \ge 1$) equals 
$(q-1)^{2} \cdot q^{2(k-1)}$ if $q$ splits in $K$, and is equal to 
$(q^{2}-1) \cdot q^{2(k-1)}$ if $q$ is inert in $K$.  (Recall that 
$K$ denotes the quadratic imaginary field that contains the order 
by which $E$ has CM.)

\section{Further remarks on the error terms}\label{Section 8}

Wu \cite{W2012} has simplified the calculations involved in the 
proofs of Theorems \ref{Theorem 1.1} and \ref{Theorem 1.2}, 
and improved on the error terms given by us. 
In the unconditional CM case, Kim \cite{K2012} has made further
improvements  by using a different approach, namely
using class field theory and a Bombieri-Vinogradov type theorem for
number fields due to Huxley \cite{H1971}.

More precisely, in the case where $E$ has CM, Wu obtained an error
term of size \linebreak 
$\Oh_E((\log X)^{-1/14}))$, 
improving on our original error term 
$\Oh_E(\log\log\log X/\log\log X)$; 
in fact his method, on noting that $\deg{\L{k}}\le k^{2}$ holds in 
the CM case (cf.~\eqref{(3.1b)}), gives an error term of size 
$\Oh_E((\log X)^{-1/8})$.
In \cite{K2012} Kim greatly improved on the error term in the CM 
case, by showing that for any $A > 0$ we have, unconditionally, 
\begin{align*}
\sum_{p \le X} \expep 
   = c_E \cdot \li{X^2} \cdot  
      \Br{1 + \Oh_{E,A}\br{\frac{1}{(\log X)^{A}}}}.
\end{align*}
Regarding conditional results, Wu \cite{W2012} has shown that on 
GRH, 
\begin{align*}
\sum_{p \le X} \expep 
 = c_E \cdot \li{X^2} 
   + \Oh_{E}\br{X^{11/6}(\log X)^{1/3}},
\end{align*}
whether or not $E$ has CM.
It is worth noting that on GRH, Wu's treatment, together with
separating out supersingular primes,
gives an improved error term in the case where $E$ has CM, namely
$
\sum_{p \le X} \expep 
 = c_E \cdot \li{X^2} 
   + \Oh_{E}\br{X^{7/4}(\log X)^{1/2}}.
$

We will now outline Wu's argument by sketching a proof of this 
GRH-conditional CM result.
As in the proofs of Theorems \ref{Theorem 1.1} and 
\ref{Theorem 1.2}, the problem reduces, via partial summation and 
the Hasse bound, to showing that
\begin{align}
  \label{(8.1)}
\sums{p \le X}{p \nmid N} \frac{1}{d_p} 
= c_E\cdot \li{X} \cdot
     \Br{1 + \Oh_{E}\br{X^{-\frac{1}{4}}(\log X)^{\frac{3}{2}}} }.
\end{align}
Briefly, we set
$Y = Y(X) \defeq X^{\frac{1}{4}}(\log X)^{-\frac{1}{2}}$.
As in the proofs of Theorems \ref{Theorem 1.1} and 
\ref{Theorem 1.2}, we use the fact that 
$\frac{1}{d_p} = \sum_{hj \mid d_p} \frac{\mu(h)}{j}$ 
to obtain
\begin{align*}
\sums{p \le X}{p \nmid N} \frac{1}{d_p}
 = \Br{\sum_{k \le Y} + \sum_{Y < k \le 2\sqrt{X}}} 
      \sum_{hj = k} \frac{\mu(h)}{j} 
       \sums{p \le X}{\text{$p \nmid N$, $k \mid d_p$}} 1. 
\end{align*}
In considering the sum over $Y < k \le 2\sqrt{X}$, we note, as we 
did in the proof of Lemma \ref{Lemma 3.6}(b), that only the primes 
of ordinary reduction contribute.
Then, again as in the proof of Lemma \ref{Lemma 3.6}(b), we use the 
trivial estimate
$
 |S_j(X;D,k)| \ll \frac{X}{k^2} + \frac{\sqrt{X}}{k}.
$
Thus,
\begin{align*}
 \sum_{Y < k \le 2\sqrt{X}}\sum_{hj = k} \frac{\mu(h)}{j} 
    \sums{p \le X}{\text{$p \nmid N$, $k \mid d_p$}} 1 
  \ll
 \sum_{y < k \le 2\sqrt{X}}
  \sumss{p \le X}{\text{$p \nmid N$, $k \mid d_p$}}{a_p \ne 0}  1  
 \ll \frac{X}{Y} + \sqrt{X}\log X.
\end{align*}
We use Lemma \ref{Lemma 3.3}(b) to handle the sum over $k \le Y$, 
obtaining
\begin{align*}
 \sum_{k \le Y}\sum_{hj = k} \frac{\mu(h)}{j} 
       \sums{p \le X}{\text{$p \nmid N$, $k \mid d_p$}} 1
  & = \li{X}
       \sum_{k \le y} \frac{1}{\deg{\L{k}}} 
        \sum_{hj = k} \frac{\mu(h)}{j} 
    + \Oh\br{Y\sqrt{X}\log(XN)}.
\end{align*}
We complete the sum, applying Lemma \ref{Lemma 3.4} to bound the 
error that arises.
Combining everything, we obtain
$
\sums{p \le X}{p \nmid N} \frac{1}{d_p} 
  = c_E\cdot \li{X} 
     \br{1 
         + \Oh_E\br{\frac{Y(\log X)^2}{\sqrt{X}} 
               + \frac{\log X}{Y}}}.
$
Since 
$Y = X^{\frac{1}{4}}(\log X)^{-\frac{1}{2}}$, this gives
\eqref{(8.1)}. 

\section*{Acknowledgements} 

For providing corrections and useful comments, we are grateful to 
the anonymous referee, Valentin Blomer, Adam Felix, Nathan Jones, 
Sungjin Kim, and Jie Wu. 

\bibliographystyle{article}

\end{document}